\documentclass{article}
\usepackage[utf8]{inputenc}
\usepackage[unicode]{hyperref}
\usepackage{amsmath,amsfonts,amssymb,amsthm}
\usepackage{fullpage}
\usepackage{indentfirst}

\hypersetup{
	colorlinks=true,
	linkcolor=blue,
	citecolor=green,
	urlcolor=black
}

\newcommand{\dd}[1]{\,\mathrm{d}{#1}}						% math mode
\newcommand{\N}{\mathbb{N}}											% math mode
\newcommand{\Z}{\mathbb{Z}}											% math mode
\newcommand{\Q}{\mathbb{Q}}											% math mode
\newcommand{\R}{\mathbb{R}}											% math mode
\newcommand{\norm}[1]{\lVert #1 \rVert}					% math mode
\newcommand{\lrsptext}[1]{\quad\text{#1}\quad}	% math mode
\newcommand{\lsptext}[1]{\quad\text{#1}\ }			% math mode
\newcommand{\rsptext}[1]{\ \text{#1}\quad}			% math mode
\newcommand{\sptext}[1]{\ \text{#1}\ }					% math mode
\newcommand{\ES}{Erd\H{o}s-Sur\'anyi }
\newcommand{\RS}{Roth-Szekeres }
\DeclareMathOperator{\li}{li}

\newtheorem{theorem}{Theorem}
\newtheorem{lemma}{Lemma}
\newtheorem{proposition}{Proposition}

\newtheorem{corollary}{Corollary}

\author{Liam Baker and Stephan Wagner\footnote{This work was supported by the National Research Foundation of South Africa, grant 96236.} \\ [5mm]
\normalsize Department of Mathematical Sciences \\
\normalsize Stellenbosch University, Private Bag X1 \\
\normalsize Matieland 7602, South Africa}
\title{\ES sequences and trigonometric integrals}
\date{}

\begin{document}\maketitle

\begin{abstract}

We study representations of integers as sums of the form $\pm a_1\pm a_2\pm \dotsb \pm a_n$, where $a_1,a_2,\ldots$ is a prescribed sequence of integers. Such a sequence is called an \ES sequence if every integer can be written in this form for some $n\in\N$ and choices of signs in infinitely many ways. We study the number of representations of a fixed integer, which can be written as a trigonometric integral, and obtain an asymptotic formula under a rather general scheme due to Roth and Szekeres. Our approach, which is based on Laplace's method for approximating integrals, can also be easily extended to find higher-order expansions. As a corollary, we settle a conjecture of Andrica and Iona\c{s}cu on the number of solutions to the signum equation $\pm 1^k \pm 2^k \pm \dotsb \pm n^k = 0$.
\end{abstract}

\section{Introduction}

\subsection{\ES sequences and solutions to signum equations}

A sequence of positive integers $(a_n)_{n=1}^\infty$ is called an \ES sequence if every integer can be written in the form $\pm a_1\pm a_2\pm \dotsb \pm a_n$ for some $n\in\N$ and choices of signs $+$ and $-$, in infinitely many ways. Representations of this kind were first studied systematically by Erd\H{o}s and Sur\'anyi \cite{erdos1959uber}, who provided sufficient conditions for a sequence of integers to have this property (that cover e.g. the sequence of primes). 

\medskip

The sequence of $k$-th powers, which will be of particular interest to us, was shown to be an \ES sequence by Mitek \cite{mitek1980generalization} and later independently by Bleicher \cite{bleicher1996prielipp}, who also discusses the behaviour of the minimal choice of $n$. Drimbe \cite{drimbe1988generalization} showed that generally, any sequence $(p(n))_{n=1}^\infty$ where $p(n)\in\Q[n]$ takes an integer value whenever $n\in\Z$, and $\gcd\{p(n)\mid n\in\Z\} =1$, is an \ES-sequence, and this result was rediscovered more recently by Yu \cite{yu2002signed} and also generalised further by Boulanger and Chabert \cite{boulanger2004representation} (to the ring of algebraic integers over a cyclotomic field), by Chen and Chen \cite{chen2012weighted} (to weights other than $\pm 1$),  and again by Chen and Chen \cite{chen2013signed} (who provided a necessary and sufficient condition for arbitrary sequences of integers).

\medskip

For an \ES sequence $\mathbf{a} =(a_n)_{n=1}^\infty$, the signum equation of $\mathbf{a}$ is $\pm a_1\pm a_2\pm \dotsb \pm a_n =0$, and for a fixed $n\in\N$, a solution to the signum equation is a choice of $+$ and $-$ such that the equation holds. We denote the number of solutions to the signum equation of $\mathbf{a}$ by $S_{\mathbf{a}}(n)$, and more generally the number of representations of an integer $k$ as $\pm a_1\pm a_2\pm \dotsb\pm a_n$ by $S_\mathbf{a}(n,k)$. In \cite{AndricaTomescu02} it is shown that the number of solutions to the signum equation can be given by the following integral formula:
\begin{equation} \label{eq:signumInt} S_{\mathbf{a}}(n) =\frac{2^n}{2\pi} \int_0^{2\pi} \prod_{i=1}^n \cos(a_i t) \dd{t}, \end{equation}
which follows from expanding each cosine into a sum of exponentials, multiplying out and using the fact that for $m\in\Z$, $\int_0^{2\pi} \exp(i m t) \dd{t}$ equals $2\pi$ if $m=0$ and $0$ if $m\neq 0$. From this, it can be easily seen that the number of representations of $k$ as $\pm a_1\pm a_2\pm \dotsb\pm a_n$ is given by $S_\mathbf{a}(n,k) =S_{\mathbf{a}^\prime}(n+1)/2$ where $\mathbf{a}^\prime =(k,a_1,a_2,\dotsc)$, as was shown in \cite{AndricaVacaretu06}.

\medskip

Andrica and Tomescu \cite{AndricaTomescu02} conjectured that the number of solutions to the signum equation in the case $a_i = i$ is asymptotically equal to $\sqrt{6/\pi} \cdot n^{-3/2} 2^n$, which was recently proved by Sullivan \cite{Sullivan13}. The related question of representing numbers as sums of the form $\sum_{k=-n}^n \epsilon_k k$ with $\epsilon_k \in \{0,1\}$ (and determining the asymptotic number of representations) was also studied in several papers, see van Lint \cite{lint}, Entringer \cite{entringer}, Clark \cite{clark}, and Louchard and Prodinger \cite{LouchardProdinger2010}. Prodinger \cite{Prodinger1982} determined an asymptotic formula for the number of ways to partition the set $\{1,2,\ldots,n\}$ into two subsets of equal cardinality and sum (note that representations of zero of the form $0 = \pm 1\pm 2\pm \dotsb \pm n$ correspond exactly to partitions of this type, where however the cardinalities are not necessarily equal). The asymptotic behaviour of an integral similar to the one in \eqref{eq:signumInt} (but with sines rather than cosines) was studied recently in \cite{gaither2015}. 

\medskip

A more general conjecture in the case $a_n = n^k$ was recently formulated by Andrica and Iona\c scu \cite{AndricaIonascu14}: namely, that (for $n \equiv 0,3 \bmod 4$)
\begin{equation}\label{eq:andrica_ionascu}
S_{\mathbf{a}}(n)  \sim \sqrt{\frac{2(2k+1)}{\pi}} \cdot \frac{2^n}{n^{k+1/2}}.
\end{equation}
The main theorem of this paper establishes an asymptotic formula for sequences $(a_n)_{n=0}^\infty$ that belong to an analytic scheme due to Roth and Szekeres (see the following section). The conjecture of Andrica and Iona\c scu will be included as a special case. It will also follow that all these sequences are \ES sequences.

\subsection{\RS sequences}\label{sec:rs_seq}

In \cite{RothSzekeres54}, Roth and Szekeres investigated partitions into elements of a sequence $(a_n)_{n=0}^\infty$ satisfying the following conditions:

\begin{enumerate}

	\item[C1.] $\label{eq:c1} a_{n+1}\geq a_n$ for sufficiently large $n$;

	\item[C2.] $\label{eq:c2} s=\displaystyle\lim_{n\to\infty} \frac{\log a_n}{\log n}$ exists and is positive;

	\item[C3.] $\label{eq:c3} J_n =\displaystyle\inf_{(2a_n)^{-1}<t\leq 1/2} \frac{\sum_{i=1}^n \norm{a_i t}^2}{\log n} \to\infty \sptext{as} n\to\infty.$
%	\item[C4.] $\label{eq:c4} \gcd\{a_n\mid n\in\N\}=1$; and
\end{enumerate}

For brevity, we will call such sequences \RS sequences. Roth and Szekeres themselves showed that the following classes of sequences are \RS sequences:

\begin{enumerate} \label{enum:RS}

	\item $a_n =p_n$, the $n$th prime number;

	\item $a_n =f(n)$, where $f$ is a polynomial with rational coefficients taking integer values at integer places, such that $\gcd f(\Z) =\gcd\{f(n)\mid n\in\Z\} =1$ (for brevity, we will call such polynomials \emph{primitive});

	\item $a_n= f(p_n)$, $p_n$ is the $n$th prime number and $f$ is a polynomial with rational coefficients taking integer values at integer places, such that $\gcd\{n f(n)\mid n\in\Z\} =1$.

\end{enumerate}

In particular, we see that if $f$ is a primitive polynomial, then $(f(n))_{n=0}^\infty$ is both a \RS sequence and has been proved to be an \ES sequence. In fact, a corollary to the main theorem of this paper shows that all \RS sequences are indeed \ES sequences.

\section{Main Theorem and applications}

\subsection{Main Theorem and sequences of applicability}

For sequences that satisfy conditions C1--C3, we are able to provide an asymptotic formula for the integral in~\eqref{eq:signumInt}:

	\begin{theorem} \label{thm:main}

	Let $(a_n)_{n=1}^\infty$ be a \RS sequence. Then \begin{equation} \label{eq:int} \int_0^{\pi/2} \prod_{i=1}^n \cos(a_i t) \dd{t} =\frac{1}{2} \sqrt{\frac{2\pi}{\sum_{i=1}^n a_i^2}} -\frac{\sqrt{2\pi} \sum_{i=1}^n a_i^4}{8 \left( \sum_{i=1}^n a_i^2 \right)^{5/2}} +O\left(n^{-s-5/2+\epsilon} \right) \lsptext{for any} \epsilon>0. \end{equation}

	\end{theorem}

As will become clear from the proof, it would be possible to derive further terms of an asymptotic expansion.

	\begin{corollary} \label{col:main}

	If $\mathbf{a} =(a_n)_{n=0}^\infty$ is a \RS sequence, then $\mathbf{a}$ is also an \ES sequence.

	\end{corollary}
	\begin{proof}

	Let $k\in\Z$, and let $\mathbf{a}^\prime =(k,a_1,a_2,\dotsc)$. Then for $n\in\N$, the number of representations of $k$ as $\pm a_1\pm a_2\pm \dotsb\pm a_n$ is
\begin{align}
S_\mathbf{a}(n,k) &=\frac{2^n}{2\pi} \int_0^{2\pi} \cos(k t) \prod_{i=1}^n \cos(a_i t) \dd{t} =\frac{2^{n+1}}{2\pi} \int_0^\pi \cos(k t) \prod_{i=1}^n \cos(a_i t) \dd{t} \nonumber \\ 
	&=\frac{2^{n+1}}{2\pi} \int_0^{\pi/2} \cos(k t) \prod_{i=1}^n \cos(a_i t) +\cos(k(\pi-t)) \prod_{i=1}^n \cos(a_i(\pi-t)) \dd{t} \nonumber \\
&=\frac{2^{n+1}}{2\pi} \left(1+(-1)^{k+\sum_{i=1}^n a_i} \right) \int_0^{\pi/2} \cos(k t) \prod_{i=1}^n \cos(a_i t) \dd{t} \nonumber \\
&=\frac{2^n}{\sqrt{2\pi}} \left(1+(-1)^{k +\sum_{i=1}^n a_i}\right) \left[\frac{1}{\sqrt{\sum_{i=1}^n a_i^2}} -\frac{\sum_{i=1}^n a_i^4}{4\left(\sum_{i=1}^n a_i^2\right)^{5/2}} +O\left(n^{-s -5/2 +\epsilon}\right)\right] \lsptext{for all} \epsilon>0 \label{eq:genSignumInt} 
\end{align}
	by \autoref{thm:main}. Now since $\mathbf{a}$ is a \RS sequence, we know that $\sum_{i=1}^n \norm{a_i/2}^2/\log n \to\infty$ as $n\to\infty$ by Condition C3, and so in particular there are infinitely many $i\in\N$ such that $a_i$ is odd. Hence there are infinitely many $n\in\N$ such that $k+\sum_{i=1}^n a_i$ is even. For these $n$, $S_\mathbf{a}(n,k) \to\infty$ as $n\to\infty$, and hence $\mathbf{a}$ is an \ES sequence.
\end{proof}

	The following proposition further expands the applicability of the main theorem:

	\begin{proposition} \label{prop:expand} \
\begin{enumerate}

		\item If $S\subset\mathbb{N}$ has the property that $\#\{k\in S\mid k\leq n\} =O(\log n)$ as $n\to\infty$ and $\mathbf{a} =(a_n)_{n=1}^\infty$ is a \RS sequence, then the subsequence of $(a_n)_{n=1}^\infty$ consisting of all elements with indices \emph{not} in $S$ is also a \RS sequence for the same value of $s$.

		\item If $\mathbf{a} =(a_n)_{n=1}^\infty$ is a \RS sequence and is also a subsequence of a sequence $\mathbf{b} =(b_m)_{m=1}^\infty$ which satisfies conditions C1 and C2 (with a possibly different value of $s$), then $\mathbf{b} =(b_m)_{m=1}^\infty$ is also a \RS sequence (i.e. also satisfies condition C3).

 \end{enumerate}
\end{proposition}

	\begin{proof} \

\begin{enumerate}
\item Let $\mathbf{a}^\prime =(a_{n_m})_{m=1}^\infty$ denote the subsequence of $\mathbf{a}$ with indices not in $S$. It is obvious that $\mathbf{a}^\prime$ also satisfies condition C1. Moreover, for large $m$ we have $m =\#\{k\notin S\mid k\leq n_m\} =n_m+O(\log n_m)$, so $\lim_{m\to\infty} \log n_m/\log m =\lim_{m\to\infty} n_m/m =1$; thus
		\[s =\lim_{n\to\infty} \log a_n/\log n =\lim_{m\to\infty} \log a_{n_m}/\log n_m =\lim_{m\to\infty} \log a_{n_m}/\log m\]
		and so $\mathbf{a}^\prime$ also satisfies condition C2 with the same value of $s$. Finally,
\begin{align*}
\inf_{(2a_{n_m})^{-1}<t\leq 1/2} \frac{\sum_{i=1}^m \norm{a_{n_i} t}^2}{\log m} &\geq\inf_{(2a_{n_m})^{-1}<t\leq 1/2} \frac{\sum_{i=1}^{n_m} \norm{a_i t}^2\ -(n_m-m)}{\log m} \\
&=\inf_{(2a_{n_m})^{-1}<t\leq 1/2} \frac{\sum_{i=1}^{n_m} \norm{a_i t}^2+O(\log n_m)}{\log m} \\
&=\inf_{(2a_{n_m})^{-1}<t\leq 1/2} \frac{\sum_{i=1}^{n_m} \norm{a_i t}^2}{\log n_m} +O(1) \to\infty \quad\text{as}\ m\to\infty,
\end{align*} 
and so $\mathbf{a}^\prime$ also satisfies condition C3.

		\item Suppose that
		\[\lim_{n\to\infty} \frac{\log a_n}{\log n} =s_1 \quad \text{and} \quad \lim_{m\to\infty} \frac{\log b_m}{\log m} =s_2,\]
		and that $\mathbf{a} =(a_n)_{n=1}^\infty$ is the subsequence $(b_{m_n})_{n=1}^\infty$ of $\mathbf{b}$. It remains to show that $\mathbf{b}$ also satisfies Condition C3. Now for  all $M\in\mathbb{N}$ let $n(M)$ be the smallest $n\in\mathbb{N}$ such that $m_n\geq M$ (and thus $a_n = b_{m_n} \geq b_M > b_{m_{n-1}}$). Then
		\[\mspace{-18mu}\frac{s_2}{s_1} = \lim_{m\to\infty} \frac{\log b_m}{\log m} \bigg( \lim_{n\to\infty} \frac{\log a_n}{\log n} \bigg)^{-1} = \lim_{n\to\infty} \frac{\log b_{m_n}}{\log m_n}  \bigg( \lim_{n\to\infty} \frac{\log a_n}{\log n} \bigg)^{-1} =\lim_{n\to\infty} \frac{\log n}{\log m_n},\]
and so $\lim_{M\to\infty} \frac{\log n(M)}{\log M} =\frac{s_2}{s_1}$. Finally,
\begin{align*}
\inf_{(2b_M)^{-1}<t\leq 1/2} \frac{\sum_{m=1}^M \norm{b_m t}^2}{\log M} &\geq \inf_{(2a_{n(M)})^{-1}<t\leq 1/2} \frac{\sum_{m=1}^M \norm{b_m t}^2}{\log M} \\
&\geq \inf_{(2a_{n(M)})^{-1}<t\leq 1/2} \frac{\sum_{k=1}^{n(M)-1} \norm{b_{m_k} t}^2}{\log M} \\
&=\inf_{(2a_{n(M)})^{-1}<t\leq 1/2} \frac{\big(\sum_{k=1}^{n(M)} \norm{a_k t}^2 \big) -1}{\log n(M)} \cdot \frac{\log n(M)}{\log M} \\
&\sim \frac{s_2}{s_1} \cdot  \inf_{(2a_{n(M)})^{-1}<t\leq 1/2} \frac{\big(\sum_{k=1}^{n(M)} \norm{a_k t}^2 \big) -1}{\log n(M)} \to\infty
\end{align*}
as $M\to\infty$, and so $\mathbf{b}$ also satisfies Condition C3.
	\end{enumerate}
\end{proof}

The first part of the proposition shows in particular that removing finitely many elements from a Roth-Szekeres sequence still yields a Roth-Szekeres sequence (by similar arguments, this is also true if finitely many elements are added). The second part shows for instance that sequences of the form $a_n = \lfloor n^s \rfloor$ for arbitrary rational numbers $s$ are also Roth-Szekeres sequences, since the sequence of numbers of the form $\lfloor n^{p/q} \rfloor$ is a subsequence of the sequence of all $p$-th powers.

\subsection{Applications of the main theorem to more specific sequences}

	\subsubsection{Polynomial-like sequences}

As an application of \autoref{thm:main}, consider the case when $\mathbf{a} =(a_n)_{n=0}^\infty$ has the asymptotic expansion $a_n=\alpha n^s+\beta n^{s-1}+O(n^{s-2})$ for some real numbers $\alpha,\beta,s$, where $\alpha > 0$ and $s > 0$. Then
		\[a_n^2 =\alpha^2 n^{2s} +2\alpha\beta n^{2s-1} +O(n^{2s-2}) \rsptext{and} a_n^4 =\alpha^4n^{4s} +O\left(n^{4s-1}\right),\]
%=\alpha^2(2s)!\binom{n}{2s} +\alpha(2s-1)!(s(2s-1)\alpha +2\beta) \binom{n}{2s-1} +O\left(n^{2s-2}\right)\]
%		\[ =\alpha^4(4s)!\binom{n}{4s} +O\left(n^{4s-1}\right)\]
so that
\[\sum_{i=1}^n a_i^2 = \frac{\alpha^2}{2s+1} n^{2s+1} +\left( \frac{\alpha^2}{2} + \frac{\alpha\beta}{s} \right) n^{2s} +O\left(n^{2s-1}\right)\]
%\alpha^2(2s)!\binom{n+1}{2s+1} +\alpha(2s-1)!(s(2s-1)\alpha+2\beta) \binom{n+1}{2s} +O\left(n^{2s-1}\right)\]
%		\[=\alpha^2(2s)!\binom{n}{2s+1} +\alpha(2s-1)!(2\alpha s+s(2s-1)\alpha+2\beta)\binom{n}{2s} +O\left(n^{2s-1}\right)\]
%		\[\frac{\alpha^2}{2s+1} n^{2s+1} +\alpha\left(-\alpha s +(s+1/2)\alpha+\beta/s\right)n^{2s} +O\left(n^{2s-1}\right) =\]
and
		\[\sum_{i=1}^n a_i^4 =%\alpha^4 (4s)! \binom{n}{4s+1} +O\left(n^{4s}\right) =
\frac{\alpha^4}{4s+1} n^{4s+1} +O\left(n^{4s}\right).\]
%		\[\rsptext{and so} \frac{1}{\sqrt{\sum_{i=1}^n a_i^2}} =\frac{\sqrt{2s+1}}{\alpha} \left[n^{-s-1/2} -\frac{2s+1}{4} (1+2\beta/s\alpha) n^{-s-3/2}\right] +O\left(n^{-s-5/2}\right). \]
It follows by \autoref{thm:main} that
\[\int_0^{\pi/2} \prod_{i=1}^n \cos(a_i t) \dd{t} =\frac{\sqrt{2\pi(2s+1)}}{2\alpha n^{s+1/2}} \left[1 -\frac{(2s+1)(s(3s+1)\alpha +(4s+1)\beta)}{2s(4s+1)\alpha n} \right] +O\left(n^{-s-5/2+\epsilon}\right) \sptext{for any} \epsilon>0.\]
In fact, in following the proof of \autoref{thm:main} in \autoref{sec:mainProof} for this sequence, we can see that $\epsilon$ may be set to zero.

In the special case that $\mathbf{a}$ is the polynomial sequence $a_n =n^s$, which is an \ES sequence as remarked earlier, we obtain the following result:
		\[S_\mathbf{a}(n) =\left(1+(-1)^{\sum_{i=1}^n i^s}\right) \sqrt{\frac{2s+1}{2\pi}} \frac{2^n}{n^{s+1/2}} \left[1-\frac{(2s+1)(3s+1)}{2(4s+1)n} +O\left(\frac{1}{n^2}\right) \right],\]
which in particular proves the asymptotic formula~\eqref{eq:andrica_ionascu} conjectured by Andrica and Iona\c{s}cu.

 If we only have the weaker property that $a_n \sim \alpha n^s$, it still follows that
\[\int_0^{\pi/2} \prod_{i=1}^n \cos(a_i t) \dd{t}  \sim \frac{\sqrt{2\pi(2s+1)}}{2\alpha n^{s+1/2}}.\]
For example, if $\mathbf{a}$ is the sequence of square-free numbers, for which it is well known that $a_n \sim \pi^2 n/6$ (this is a \RS sequence, e.g. by \autoref{prop:expand} applied to the sequence of primes, which are all square-free),
we get
\[\int_0^{\pi/2} \prod_{i=1}^n \cos(a_i t) \dd{t}  \sim \frac{3\sqrt{6}}{(\pi n)^{3/2}}\]
and a corresponding asymptotic formula for the number of solutions to the signum equation.

	\subsubsection{Polynomials in primes}
		Let us also consider the case when $f$ is a polynomial satisfying the properties mentioned in \autoref{sec:rs_seq}, $f(n) =\alpha n^s +O\left(n^{s-1}\right)$, and $a_n =f(p_n)$ where $p_n$ is the $n$th prime number. We will use the following standard lemma, whose proof is given for completeness:
		\begin{lemma}
			
			If $q(x) =\sum_{i=0}^s c_i x^i$ is a polynomial with $c_s>0$, then $\sum_{i=1}^n q(p_i) \sim\frac{c_s}{s+1} n^{s+1}(\log n)^s$.
			
		\end{lemma}
		\begin{proof}
			
			We use the ideas described in Section 2.7 of \cite{Bach96}. Let $\pi(x)$ be the prime counting function, $\li(x) =\int_2^x \dd{t}/\log t$ be the logarithmic integral, and define $\epsilon(x) =\pi(x) -\li(x)$ which is $o(t/\log t)$ by the prime number theorem. Then we write the above sum as a Stieltjes integral, where $b<2$:
			\[\sum_{i=1}^n q(p_i) =\sum_{p\leq p_n} q(p) =\int_b^{p_n} q(t) \dd{\pi(t)} =\int_b^{p_n} q(t) \dd{(\li(t) +\epsilon(t))}.\]
			 
			Now we note that $\dd{(\li(t))} =\dd{t}/\log t$ and we perform integration by parts on $\int_b^{p_n} q(t) \dd{\epsilon(t)}$, giving
			\[\sum_{i=1}^n q(p_i) =\int_b^{p_n} \frac{q(t) \dd{t}}{\log t} +[q(t)\epsilon(t)]_b^{p_n} -\int_b^{p_n} \epsilon(t) q^\prime(t)\dd{t}.\]
			Letting $b\to2^-$ and noting that $\epsilon(b)\to0$, this becomes
			\[\sum_{i=1}^n q(p_i) =\int_2^{p_n} \frac{q(t) \dd{t}}{\log t} +q(p_n)\epsilon(p_n) -\int_2^{p_n} \epsilon(t) q^\prime(t) \dd{t} =\sum_{i=0}^s c_i\left[\int_2^{p_n} \frac{t^i \dd{t}}{\log t} +p_n^i \epsilon(p_n) -\int_2^{p_n} \epsilon(t) it^{i-1} \dd{t}\right].\]
			In the $i$th summand, the first integral is an example of an exponential integral and has asymptotic expansion $\frac{p_n^{i+1}}{(i+1)\log p_n}(1+O(1/\log p_n))$, whereas using the asymptotic bound on $\epsilon(t)$ it is easily seen that the other two terms are $o(p_n^{i+1}/\log p_n)$. Hence we have the desired asymptotic expansion, using the asymptotic formula $p_n\sim n/\log n$ that follows from the prime number theorem:
			\[\sum_{i=1}^n q(p_i) \sim\frac{c_s}{s+1} \frac{p_n^{s+1}}{\log p_n} \sim\frac{c_s}{s+1} n^{s+1}(\log n)^s\]
\end{proof}
Applying this lemma, we get that
		\[\sum_{i=1}^n a_i^2 =\sum_{i=1}^n f(p_n)^2 \sim\frac{\alpha^2}{2s+1} n^{2s+1}(\log n)^{2s}.\]
		Thus \autoref{thm:main} gives us
		\[\int_0^{\pi/2} \prod_{i=1}^n \cos(a_i t) \dd{t} \sim\frac{1}{2}\sqrt{\frac{2\pi}{\frac{\alpha^2}{2s+1} n^{2s+1} (\log n)^{2s}}} =\frac{\sqrt{2\pi(2s+1)}}{2\alpha} \frac{1}{n^{s+1/2} (\log n)^s},\]
and so
\[S_\mathbf{a}(n) \sim\frac{\sqrt{2s+1}}{\sqrt{2\pi}\alpha} \frac{2^{n+1}}{n^{s+1/2}(\log n)^s} \lsptext{as} n\to\infty \sptext{and} \sum_{i=1}^n f(p_n) \sptext{is even.}\]
Specifically, if $\mathbf{a}$ is the sequence of primes,
\[S_\mathbf{a}(n) \sim \sqrt{\frac{6}{\pi}} \frac{2^n}{n^{3/2} \log n} \lsptext{as} n\to\infty \sptext{for odd} n.\]

\section{Proof of Main Theorem} \label{sec:mainProof}

	To prove \autoref{thm:main}, we will require the following lemmas; in both of them, we assume that the sequence $\mathbf{a}$ satisfies Condition C2. 

	\begin{lemma} \label{lem:s24}

		If $b>0$, then $\sum_{i=1}^n a_i^b =O\left(n^{bs+1+\epsilon} \right)$ and $\left(\sum_{i=1}^n a_i^b\right)^{-1} =O\left(n^{-bs-1+\epsilon}\right)$ for any $\epsilon>0$. As corollaries, we have the following:

		If $b>d>0$ and $c\in\R$, then \begin{equation} \frac{\sqrt[d]{\sum_{i=1}^n a_i^d}}{(\log n)^c\sqrt[b]{\sum_{i=1}^n a_i^b}} \to\infty \text{ as } n\to\infty,\lrsptext{and} \frac{\sqrt[d]{\sum_{i=1}^n a_i^d}}{(\log n)^c a_n} \to\infty \text{ as } n\to\infty. \end{equation}

	\end{lemma}

\begin{proof}
This follows immediately from the fact that for any $\delta > 0$, $i^{s-\delta} < a_i < i^{s+\delta}$ for sufficiently large $i$ according to Condition C2.
\end{proof}

	\begin{lemma} \label{lem:intexp}

		If $b_n>0$ and $b_n^2 \sum_{i=1}^n a_i^2 /\log n \to\infty$ as $n\to\infty$, then for any $m\geq0$ and any fixed $\ell > 0$, we have
		\[\int_0^{b_n} \pi^mt^m \exp\left(-\frac{\pi^2t^2}{2} \sum_{i=1}^n a_i^2\right) \dd{t} =2^{(m-1)/2} \pi^{-1} \Gamma\left(\frac{m+1}{2}\right) \left(\sum_{i=1}^n a_i^2\right)^{-(m+1)/2} +O\left(n^{-\ell}\right).\]

	\end{lemma}
	\begin{proof}

		\[\int_0^{b_n} t^m \exp\left(-\frac{\pi^2t^2}{2} \sum_{i=1}^n a_i^2\right) \dd{t} =\int_0^\infty t^m \exp\left(-\frac{\pi^2t^2}{2} \sum_{i=1}^n a_i^2\right) \dd{t} -\int_{b_n}^\infty t^m \exp\left(-\frac{\pi^2t^2}{2} \sum_{i=1}^n a_i^2\right) \dd{t}.\]
		The first integral is (substituting $u= (\pi^2t^2/2) \sum_{i=1}^n a_i^2$)
\begin{align*} \int_0^\infty t^m \exp\left(-\frac{\pi^2t^2}{2} \sum_{i=1}^n a_i^2\right) \dd{t} &=
2^{(m-1)/2} \left(\pi^2 \sum_{i=1}^n a_i^2\right)^{-(m+1)/2} \int_0^\infty u^{(m-1)/2} e^{-u} \dd{u} \\
&=2^{(m-1)/2} \pi^{-m-1} \Gamma\left(\frac{m+1}{2}\right) \left(\sum_{i=1}^n a_i^2\right)^{-(m+1)/2},
\end{align*}
		and by a similar procedure the second integral can be written as follows:
		\[\int_{b_n}^\infty t^m \exp\left(-\frac{\pi^2t^2}{2} \sum_{i=1}^n a_i^2\right) \dd{t} =2^{(m-1)/2} \pi^{-m-1} \left(\sum_{i=1}^n a_i^2\right)^{-(m+1)/2} \int_{x_n}^\infty u^{(m-1)/2} e^{-u} \dd{u},\]
		where $x_n =\pi^2b_n^2 \sum_{i=1}^n a_i^2/2 \to\infty$ as $n\to\infty$. Now for $u$ sufficiently large, $u^{(m-1)/2} \leq e^{u/2}$, so that for $n$ sufficiently large,
\[0 \leq \int_{x_n}^\infty u^{(m-1)/2} e^{-u} \dd{u} \leq \int_{x_n}^\infty e^{-u/2} \dd{u} =2 e^{-x_n/2} =O\left(n^{-\ell}\right).\]
Here, the last estimate follows from the assumption made on $b_n$, which implies that $x_n/\log n \to \infty$ as $n \to \infty$. This completes the proof.
\end{proof}

	\bigskip
	Now we are ready to prove \autoref{thm:main}:

	\begin{proof}[Proof of \autoref*{thm:main}]

		We assume throughout that $n$ is large, and thus that $a_n$ is large positive and not less than $a_i$ for $i<n$. We rewrite the integral in \eqref{eq:int} as follows: \begin{equation} \int_0^{\pi/2}\prod_{i=1}^n \cos(a_i t)\dd{t} =\pi\int_0^{1/(2a_n)} \prod_{i=1}^n \cos(a_i\pi t)\dd{t} +\pi\int_{1/(2a_n)}^{1/2} \prod_{i=1}^n \cos(a_i\pi t)\dd{t} =I_1+I_2. \end{equation}

		The second integral, $I_2$, can be estimated as follows, making use of the simple inequality $|\cos(\pi x)|\leq \exp(-\pi^2\|x\|^2/2)$ that is valid for all real $x$:

\begin{align*}
\left|\int_{1/(2a_n)}^{1/2} \prod_{i=1}^n \cos(a_i\pi t)\dd{t}\right| &\leq\int_{1/(2a_n)}^{1/2} \prod_{i=1}^n |\cos(a_i\pi t)| \dd{t} \leq\int_{1/(2a_n)}^{1/2} \prod_{i=1}^n \exp\left(-\frac{\pi^2}{2} \|a_i t\|^2\right) \dd{t} \\
&\leq\int_{1/(2a_n)}^{1/2} \exp\left(-\frac{\pi^2}{2} \sum_{i=1}^n \|a_i t\|^2\right) \dd{t}
\leq\int_{1/(2a_n)}^{1/2} \exp\left(-\frac{\pi^2}{2}J_n\log n \right) \dd{t} \\ &=\left[\frac{1}{2}-\frac{1}{2a_n}\right] \exp\left(-\frac{\pi^2}{2}J_n\log n\right) <\frac{1}{2} n^{-\pi^2 J_n/2}.
\end{align*}
Since $J_n\to\infty$ as $n\to\infty$ by condition C3, it follows that

\begin{equation} \label{eq:firstEstimate} I_2 =\pi\int_{1/2a_n}^{1/2}\prod_{i=1}^n \cos(a_i\pi t)\dd{t} =O(n^{-\ell}) \lsptext{for any} \ell>0. \end{equation}

		We now split up the first integral, $I_1$, again:
		\[\pi\int_0^{1/(2a_n)} \prod_{i=1}^n \cos(a_i\pi t)\dd{t} =\pi\int_0^{b_n} \prod_{i=1}^n \cos(a_i\pi t)\dd{t} +\pi\int_{b_n}^{1/(2a_n)} \prod_{i=1}^n \cos(a_i\pi t)\dd{t} =I_3+I_4\] where $b_n\in(0,1/(2a_n))$ will be chosen later. $I_4$ can be estimated as before (note that $\|a_i t\| = a_i t$ for $0 \leq t \leq 1/(2a_n) \leq 1/(2a_i)$):
\begin{align*}
\left|\int_{b_n}^{1/2a_n} \prod_{i=1}^n \cos(a_i\pi t) \dd{t} \right| &\leq\int_{b_n}^{1/2a_n} \exp\left(-\frac{\pi^2t^2}{2} \sum_{i=1}^n a_i^2\right)\dd{t} <\int_{b_n}^{1/2a_n} \exp\left(-\frac{\pi^2b_n^2}{2} \sum_{i=1}^n a_i^2\right)\dd{t} \\
&<\int_{0}^{1/2a_n} \exp\left(-\frac{\pi^2b_n^2}{2} \sum_{i=1}^n a_i^2\right)\dd{t} =\frac{1}{2a_n} \exp\left(-\frac{\pi^2b_n^2}{2} \sum_{i=1}^n a_i^2\right).
\end{align*}
		We then have the following estimate: \begin{equation} \label{eq:secondEstimate} I_4 = O(n^{-\ell}) \text{ for any } \ell>0, \quad\text{provided that } b_n^2 \sum_{i=1}^n a_i^2 /\log n\to\infty \text{ as } n\to\infty. \end{equation}
Now we use the Taylor expansion
$$\log \cos x = - \frac{x^2}{2} - \frac{x^4}{12} + O(x^6),$$
which gives
\begin{align}
\prod_{i=1}^n \cos(a_i\pi t) &=\exp\left(-\frac{\pi^2t^2}{2}\sum_{i=1}^n a_i^2 -\frac{\pi^4t^4}{12} \sum_{i=1}^n a_i^4 + O \left( t^6 \sum_{i=1}^n a_i^6 \right) \right) \nonumber \\
&= \exp\left(-\frac{\pi^2t^2}{2}\sum_{i=1}^n a_i^2\right) \left( 1 -\frac{\pi^4t^4}{12} \sum_{i=1}^n a_i^4 + O\left( t^6 \sum_{i=1}^n a_i^6 + t^8 \left(\sum_{i=1}^n a_i^4\right)^2\right) \right) \label{eq:thirdEstimate}
\end{align}
for $|t| \leq b_n$, provided that $b_n^4 \sum_{i=1}^n a_i^4 \to 0$ and $b_n^6 \sum_{i=1}^n a_i^6 \to 0$ as $n\to\infty$ (in fact, the former implies the latter).
Thus by \autoref{lem:intexp},
\begin{align*}
\int_0^{b_n} &\prod_{i=1}^n \cos(a_i\pi t) \dd{t} \\
&=\int_0^{b_n} \exp\left(-\frac{\pi^2t^2}{2} \sum_{i=1}^n a_i^2\right) \dd{t} - \sum_{i=1}^n a_i^4 \int_0^{b_n} \frac{\pi^4t^4}{12} \exp\left(-\frac{-\pi^2t^2}{2} \sum_{i=1}^n a_i^2\right) \dd{t} \\
&\qquad+O\left(\sum_{i=1}^n a_i^6 \int_0^{b_n} t^6 \exp\left(-\frac{\pi^2t^2}{2} \sum_{i=1}^n a_i^2\right) \dd{t} + \left(\sum_{i=1}^n a_i^4\right)^2 \int_0^{b_n} t^8 \exp\left(-\frac{\pi^2t^2}{2} \sum_{i=1}^n a_i^2\right) \dd{t}\right) \\
&=\frac{1}{\sqrt{2}\pi} \Gamma(1/2) \left(\sum_{i=1}^n a_i^2\right)^{-1/2} -\frac{2^{3/2}}{12\pi} \Gamma(5/2) \sum_{i=1}^n a_i^4 \left(\sum_{i=1}^n a_i^2\right)^{-5/2} \\
&\qquad+O\left(\sum_{i=1}^n a_i^6 \left(\sum_{i=1}^n a_i^2\right)^{-7/2} + \left(\sum_{i=1}^n a_i^4\right)^2 \left(\sum_{i=1}^n a_i^2\right)^{-9/2}\right) +O\left(n^{-\ell}\right)
\end{align*}
		for any $\ell>0$. Now by \autoref{lem:s24}, we have
\[\sum_{i=1}^n a_i^6 \left(\sum_{i=1}^n a_i^2\right)^{-7/2} =O\left(n^{-s-5/2+\epsilon}\right) \lrsptext{and} \left(\sum_{i=1}^n a_i^4\right)^2 \left(\sum_{i=1}^n a_i^2\right)^{-9/2} =O\left(n^{-s-5/2+\epsilon}\right)\]
for any $\epsilon>0$. 
%Hence
%		\[\int_0^{b_n} \exp\left(-\frac{\pi^2t^2}{2} \sum_{i=1}^n a_i^2 -\frac{\pi^4t^4}{12} \sum_{i=1}^n a_i^4\right) \dd{t} =\frac{1}{\sqrt{2\pi}} \left(\sum_{i=1}^n a_i^2\right)^{\mspace{-6mu} -1/2} \mspace{-6mu}-\frac{1}{4\sqrt{2\pi}} \sum_{i=1}^n a_i^4 \left(\sum_{i=1}^n a_i^2\right)^{\mspace{-6mu}-5/2} +O\left(n^{-s-5/2+\epsilon}\right)\]
%		for any $\epsilon>0$. 
%This combined with \eqref{eq:thirdEstimate} shows that
Thus it follows that
		\begin{equation} \label{eq:fifthEstimate} I_3 =\frac{1}{2} \sqrt{\frac{2\pi}{\sum_{i=1}^n a_i^2}} -\frac{\sqrt{2\pi}}{8} \frac{\sum_{i=1}^n a_i^4}{\left(\sum_{i=1}^n a_i^2\right)^{5/2}} +O\left(n^{-s-5/2+\epsilon}\right) \lsptext{for any} \epsilon>0. \end{equation}
Finally, combining \eqref{eq:firstEstimate}, \eqref{eq:secondEstimate} and \eqref{eq:fifthEstimate} we arrive at the following second-order approximation for our initial integral:
		\begin{equation} \label{eq:finalResult} \int_{-\pi/2}^{\pi/2} \prod_{i=1}^n \cos(a_i t)\dd{t} =I_3+I_4+I_2 =\frac{1}{2} \sqrt{\frac{2\pi}{\sum_{i=1}^n a_i^2}} -\frac{\sqrt{2\pi}\sum_{i=1}^n a_i^4}{8\left(\sum_{i=1}^n a_i^2\right)^{5/2}}+O(n^{-s-5/2+\epsilon}) \lsptext{for any} \epsilon>0.\end{equation}

		The only issue which yet remains is the existence of a sequence $(b_n)_{n=1}^\infty$ satisfying the conditions imposed on it in \autoref{lem:intexp}, \eqref{eq:secondEstimate} and \eqref{eq:thirdEstimate}. Using \autoref{lem:s24}, it is easy to see that $b_n =n^{-s-1/3}$ satisfies these conditions, and hence our proof is complete.
\end{proof}

\bibliographystyle{abbrv}
\bibliography{bib}

\end{document}